%% file: main.tex
\documentclass[a4paper, reqno]{amsart}
\usepackage[english]{babel}
\usepackage[utf8]{inputenc}
\usepackage[T1]{fontenc}
\usepackage{lmodern}
\usepackage{csquotes}
\usepackage{amsmath,amssymb,amsthm}
\usepackage{mathtools}
\usepackage{tikz}
\usetikzlibrary{calc, arrows.meta, patterns, positioning}
\usepackage[font=footnotesize, margin=1em]{caption}
\usepackage{hyperref}
\usepackage{subcaption}
\usepackage{float}
\usepackage{cite}

\definecolor{sienne}{RGB}{136, 45, 23}
\hypersetup{
colorlinks = true,
linkcolor = sienne
}
\allowdisplaybreaks[1]

\newtheorem{theorem}{Theorem}[section]
\newtheorem{prop}[theorem]{Proposition}
\newtheorem{corollary}[theorem]{Corollary}
\newtheorem{lemma}[theorem]{Lemma}
\theoremstyle{definition}
\newtheorem{definition}[theorem]{Definition}
\theoremstyle{remark}
\newtheorem{remark}[theorem]{Remark}

\newcommand{\puncfootnote}[1]{\kern-0.2em\footnote{#1}}

\DeclareMathOperator{\distance}{distance}

\DeclareMathOperator{\supp}{supp}

\DeclareMathOperator{\ran}{Range}
\DeclareMathOperator{\abscissa}{abscissa}

\numberwithin{equation}{section}

\newcommand{\diff}[1][-3]{\ensuremath{\mathop{}\mkern#1mu\mathrm{d}}}
\def\eu{\ensuremath{\mathrm{e}}}
\def\iu{\ensuremath{\mathrm{i}}}
\newcommand{\set}{\mathbb}

\mathtoolsset{mathic}
\newcommand*{\noic}{\sb{}\kern-\scriptspace }

\newcommand{\no}{n\rule[0.4ex]{0.4em}{0.7pt}\kern-1.08ex\textsuperscript{o}}

\begin{document}
\title{Control of the Grushin equation: non-rectangular control region and minimal time}\thanks{
The first author was partially supported by the Project ``Analysis and simulation of optimal shapes - application to lifesciences'' of the Paris City Hall.
}\thanks{The second author was partially supported by the ERC advanced grant 
SCAPDE, seventh framework program, agreement \no~320845.}
\author{Michel Duprez}\address{
Sorbonne Universit\'es, UPMC Univ Paris 06, CNRS UMR 7598, Laboratoire Jacques-Louis Lions, F-75005, Paris,
France. e-mail: \texttt{mduprez@math.cnrs.fr}}
\author{Armand Koenig}\address{Université Côte d’Azur, CNRS, LJAD, France. e-mail: \texttt{armand.koenig@unice.fr}}
\date{\today}
\begin{abstract} This paper is devoted to the study of the internal null-controllability of the Grushin equation. 
We determine the minimal time of controllability for a large class of non-rectangular control region.
We establish the positive result thanks to the fictitious control method and the negative one by interpreting the associated observability inequality as an $L^2$ estimate on complex polynomials. \end{abstract}
%
%
\maketitle
\section{Setting}

Let $\Omega := (-1,1)\times \set (0,\pi)$, $\omega$ be an open subset of $\Omega$ and $T>0$.
We denote by $\partial\Omega$ the boundary of $\Omega$.
In this work, we consider the Grushin equation:
\begin{equation}\label{eq:control_problem_eng}
\left\{
\begin{array}{ll}
(\partial_t - \partial_x^2 - x^2\partial_y^2)f(t,x,y)  = \mathbf 1_\omega u(t,x,y) &  t\in [0,T], (x,y)\in \Omega,\\
f(t,x,y) = 0 &  t \in [0,T], (x,y)\in \partial\Omega,\\
f(0,x,y) = f_0(x,y)& (x,y)\in \Omega,
\end{array}
\right.
\end{equation}
where $f_0\in L^2(\Omega)$ is the initial data and $u\in L^2([0,T]\times\omega)$ the control.

Define the inner product 
\[(f,g):=\int_{\Omega}(\partial_xf\partial_xg+x^2\partial_yf\partial_yg)\diff x\diff y,\]
for all $f,g\in  C_0^{\infty}(\Omega)$,
and set
$V(\Omega):=\overline{ C_0^{\infty}(\Omega)}^{|\cdot|_{V(\Omega)}}$, with $|\cdot|_{V(\Omega)}:=(\cdot,\cdot)^{1/2}.$

For any initial data $f_0\in L^2(\Omega)$ and any control $u\in L^2([0,T]\times\Omega)$,
it is well known~\cite[Section 2]{beauchard_2014} that the Grushin equation~\eqref{eq:control_problem_eng}
admits a unique weak solution 
\begin{equation}\label{prop:well}
f\in  C([0,T];L^2(\Omega))\cap L^2(0,T;V(\Omega)).
\end{equation}

The Grushin equation~\eqref{eq:control_problem_eng} is said to be \emph{null-controllable} in time $T>0$ 
if for each initial data $f_0$ in $L^2(\Omega)$, there exists a control $u$ in $L^2((0,T)\times \Omega)$ such that the solution $f$ to System~\eqref{eq:control_problem_eng}
satisfies $f(T,x,y) = 0$ for all $(x,y)$ in $\Omega$.

\section{Bibliographical comments}\label{sec:bib com}

The null-controllability of the heat equation is well known in dimension one since 1971~\cite{fattorini_1971} and in higher dimension since 1995~\cite{lebeau_1995, fursikov_1996}: the heat equation on a $C^2$ bounded
domain is null-controllable in any non-empty open subset and in arbitrarily small time. But for degenerate parabolic equations, i.e. equations where the
laplacian is replaced by an elliptic operator that is not uniformly elliptic, we only have results for some particular equations. For instance, control properties of one and two dimensional parabolic equations where the degeneracy
is at the boundary are now understood~\cite{cannarsa_2016}. Other examples of control properties of degenerate parabolic equations that have been
investigated include some Kolmogorov-type equation~\cite{beauchard_2015a}, quadratic differential equations~\cite{beauchard_2016} and the heat equation on
the Heisenberg group~\cite{beauchard_2017} (see also references therein).

About the Grushin equation, Beauchard, Cannarsa and Guglielmi~\cite{beauchard_2014} proved in 2014 that if $\omega$ is a vertical strip that does not touch the degeneracy line $\{x=0\}$, there exists $T^\ast>0$ such that the Grushin equation is not null-controllable if
$T<T^\ast$ and that it is null-controllable if $T>T^\ast$; moreover, they proved that if $a$ is the distance between the control domain $\omega$ and the degeneracy line $\{x=0\}$, then $a^2\!/2 \leq T^\ast$. Then, Beauchard, Miller
and Morancey~\cite{beauchard_2015} proved that if $\omega$ is a vertical strip that touches the degeneracy line, null-controllability holds in arbitrarily small time, and that
if $\omega$ is two symmetric vertical strips, the minimal time is exactly
$a^2\!/2$, where $a$ is the distance between $\omega$ and $\{x=0\}$. The
minimal time if we control from the left (or right) part of the boundary was
computed by Beauchard, Dardé and Ervedoza~\cite{beauchard_2018}, and our
positive result is based on that. On the other hand, the second author proved
that if $\omega$ does not intersect a horizontal strip, then the Grushin
equation is never null-controllable~\cite{koenig_2017}, and our negative result
is proved with the methods of this reference.

So, the Grushin equation needs a minimal time for the null-controllability to
hold, as do the Kolmogorov equation and the heat equation on the Heisenberg group (see above references); a feature more expected for hyperbolic equations
than parabolic ones. Note, however, that degenerate parabolic equations are not the only parabolic equations that exhibit a minimal time of null-controllability. In dimension one, Dolecki has proved there exists a
minimal time for the punctual controllability of the heat equation to hold~\cite{dolecki73}, and parabolic systems may also present a minimal time of
null-controllability  (see e.g. \cite{ammar2014,ammar2016,D2017,khodja2017quantitative}).

Another problem we wight look at is the \emph{approximate} null-controllability. Approximate null-controllability of course holds when \emph{exact} null-controllability hold. Actually, approximate null-controllability always holds for the Grushin equation~\cite[Proposition 3]{beauchard_2014} (see also references therein) and Morancey proved approximate null-controllability also holds if we add some potential that is singular at $x=0$ \cite{morancey_sing}.

\section{Main results}

Our first result is about the null-controllability in large time if the control domain contains an $\varepsilon$-neighborhood of a path that goes from the bottom boundary to the top boundary:
\begin{theorem}[Positive result]\label{th:main_postiive}
Assume that there exists $\varepsilon>0$ and\footnote{We denote $\overline{\Omega}$ the closure of $\Omega$.} $\gamma\in C^{0}([0,1],\overline{\Omega})$ 
with $\gamma(0)\in (-1,1)\times\{0\}$ and $\gamma(1)\in (-1,1)\times\{\pi\}$ such that
\begin{equation}\label{omega_0}
\omega_0 \coloneqq \{z\in\Omega, \distance(z,\ran(\gamma))<\varepsilon\}\subset\omega,
\end{equation}
(see Fig.~\ref{fig:omega_pos_th}). Let 
\begin{equation}\label{def a}
a\coloneqq\max_{s\in[0,1]}(|\abscissa(\gamma(s))|).
\end{equation}
Then the Grushin equation~\eqref{eq:control_problem_eng} is null-controllable on $\omega$ in any time $T>a^2\!/2$.
\end{theorem}
We prove this Theorem in Section~\ref{sec:pos}.

\begin{figure}
\centering
\input{omega_pos_th}
\caption{In green, an example of function $\gamma$ and a domain $\omega_0$ for Theorem~\ref{th:main_postiive}.
If $\omega$ contains such domain $\omega_0$, 
then the Grushin equation~\eqref{eq:control_problem_eng} is null-controllable in time $T>a^2\!/2$.}
\label{fig:omega_pos_th}
\end{figure}
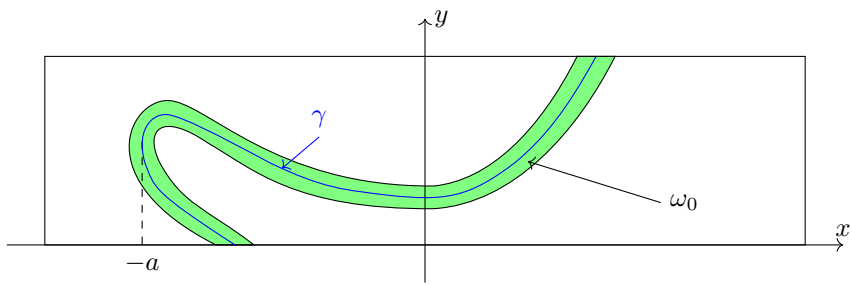

\begin{remark}\label{rk:generalized_grushin}
\begin{enumerate}
\item Theorem \ref{th:main_postiive} can be adapted for $\Omega:=(L_-,L_+)\times (0,\pi)$ and 
 the following generalized version of the Grushin equation:
\begin{equation}\label{eq:control_problem_eng q(x)}
\left\{
\begin{array}{ll}
(\partial_t - \partial_x^2 - q(x)^2\partial_y^2)f(t,x,y)  = \mathbf 1_\omega u(t,x,y) &  t\in [0,T], (x,y)\in \Omega,\\
f(t,x,y) = 0 &  t \in [0,T], (x,y)\in \partial\Omega,\\
f(0,x,y) = f_0(x,y)& (x,y)\in \Omega,
\end{array}
\right.
\end{equation}
where $q$ satisfies the following conditions:
\begin{equation*}
q(0)=0, ~~q\in C^3([L_-,L_+]),~~\inf_{(L_-,L_+)}\{\partial_xq\}>0.
\end{equation*}
Then System \eqref{eq:control_problem_eng q(x)}
is null-controllable in any time $T>T^*=q(0)^{-1}\int_0^{a}q(s)\diff s$.
Indeed, the proof of Theorem \ref{th:main_postiive} is based on the result of \cite{beauchard_2018}  which is given in the setting of the equation \eqref{eq:control_problem_eng q(x)}.
\item Since the Grushin equation~\eqref{eq:control_problem_eng} is not null-controllable when $\omega$ is the complement of a horizontal
strip~\cite{koenig_2017}, Assumption~\eqref{omega_0} is quasi optimal.
\end{enumerate}
\end{remark}

We now state our second main result:
\begin{theorem}[Negative result]\label{th:main_negative}
If for some $y_0\in(0,\pi)$ and $a>0$, the horizontal open segment $\{(x,y_0), |x|<a\}$ is disjoint from $\overline{\omega}$ (see Fig.~\ref{fig:omega_neg_th}), 
then the Grushin equation~\eqref{eq:control_problem_eng} is not null-controllable in time $T<a^2\!/2$.
\end{theorem}
We prove this Theorem in Section~\ref{sec:neg}.
Note that this Theorem stays true if we replace the domain $\Omega$ by $\set R\times (0,\pi)$. In particular, we have the following Theorem~\ref{th:main_neg_R}, which answers a question that was asked to the second author by Yves Colin de~Verdière.

\begin{figure}
\centering
\input{omega_neg_th}
\caption{In green, an example of a domain $\omega$ for Theorem~\ref{th:main_negative}. If we have a symmetric horizontal segment of length $2a$ that does not touch $\omega$ (except maybe the extremities), the Grushin equation is not null-controllable in time $T<a^2\!/2$.}
\label{fig:omega_neg_th}
\end{figure}
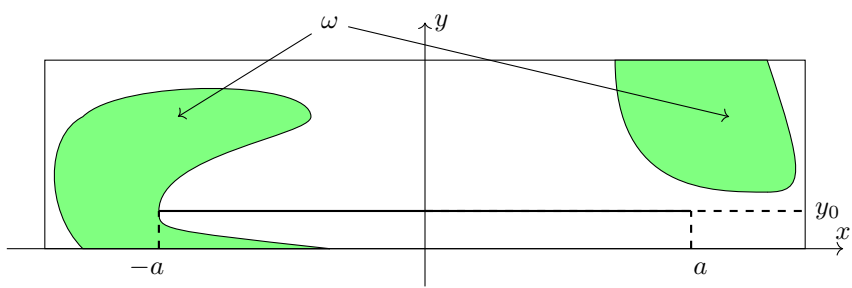

\begin{theorem}[Negative result on the whole real line]\label{th:main_neg_R}
Let $y_0\in(0,\pi)$, $f\colon\set R\to \set R_+^\star$ a continuous function that is never zero and $\omega = \{(x,y)\in \set R\times (0,\pi), |y-y_0|> f(x)\}$  (see Fig. \ref{fig:omega_R}). Then Grushin equation on $\Omega = \set R\times(0,\pi)$ is never null-controllable on $\omega$.
\end{theorem}
We sketch the proof of this Theorem in Appendix~\ref{sec:neg_R}.
Note that both of these negative results are valid if we take $y\in \set R\slash2\pi\set Z$ instead of $y\in(0,\pi)$.

\begin{figure}
\centering
\input{omega_R}
\caption{In green, the domain $\omega$ in Theorem~\ref{th:main_neg_R}. Even if when $|x|$ tends to $\infty$, the complement of $\omega$ narrows, the Grushin equation is never null-controllable.}
\label{fig:omega_R}
\end{figure}
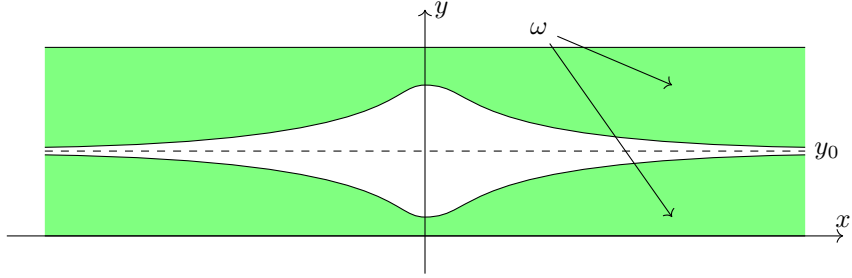

With Theorem~\ref{th:main_postiive}, we can (often) find an upper bound on the minimal time of control, and with Theorem~\ref{th:main_negative}, we can lower bound it. If these two bounds coincide, then we have the actual minimal time of control. We can prove this is the case for a large case of control domains, for instance:
\begin{corollary}\label{cor}
Let $\gamma_1$ and $\gamma_2$ be two continuous functions from $[0,\pi]$ to $(-1,1)$ such that $\gamma_1<\gamma_2$, let $\omega = \{(x,y):\gamma_1(y) < x < \gamma_2(y)\}$, and let\footnote{We denote $\gamma^+ = \max(0,\gamma)$ and $\gamma^- = \max(0,-\gamma)$.} $a = \max(\max (\gamma_2^-), \max (\gamma_1^+))$. One has:
\begin{enumerate}
\item if $T>a^2\!/2$, then the Grushin equation~\eqref{eq:control_problem_eng} is null-controllable in time $T$;
\item if $T<a^2\!/2$, then the Grushin equation~\eqref{eq:control_problem_eng} is not null-controllable in time $T$.
\end{enumerate}
\end{corollary}
We prove this Corollary in Section~\ref{sec:cor}.

\section{Comments and open problems}
Note that most of the existing results for the controllability of degenerate parabolic equations (see section~\ref{sec:bib com}) were only concerned with rectangular control domains. The reason is that these results were based on Fourier series techniques, which can only treat tensorised domains. Our results are built on the previous ones, but by adding some arguments (the fictitious control method for the positive result, and fully treating $x$ as a parameter in the negative one) we can accurately treat a large class of non-rectangular domains. But even then, there remains some open problems.

\begin{figure}
\centering
    \begin{subfigure}[t]{0.3\textwidth}
    \input{example1}
    \caption{We can't go from the bottom boundary to the top boundary while staying inside $\omega$.}
    \label{fig:example.1}
\end{subfigure}~
\begin{subfigure}[t]{0.3\textwidth}
    \input{example2}
    \caption{A pinched domain.}
    \label{fig:example.2}
\end{subfigure}~
\begin{subfigure}[t]{0.3\textwidth}
    \input{example3}
    \caption{A cave.}
    \label{fig:example.3}
\end{subfigure}
\caption{Examples of control region for which the minimal time of null-controllability is unknown.}
\end{figure}

If the control region $\omega$ if not connected, 
the positive result of Theorem~\ref{th:main_postiive} might not apply. For example in Figure~\ref{fig:example.1}, 
the negative result of Theorem~\ref{th:main_negative}
says only that the minimal time for the null-controllability of the Grushin equation is greater than $a^2\!/2$, 
but we don't know if the Grushin equation is null-controllable in any time greater than $a^2\!/2$ or not.

If the domain is ``pinched'', as in Figure~\ref{fig:example.2}, 
the open segment $\{(x,\pi/2):-1<x<1\}$ is disjoint from $\omega$ but not from $\overline{\omega}$: we are in a limit case of Theorem~\ref{th:main_negative}.
Again, we only know that the minimal time is greater than $a^2\!/2$, but we don't know if the Grushin equation is ever null-controllable at all.

For $\omega:=\{(x,y):\gamma_1(y)<x<\gamma_2(y)\}$ with
$\gamma_1,\gamma_2\in C([0,\pi];(-1,1))$ 
such that $\gamma_1<\gamma_2$,  Corollary~\ref{cor}
 determines the minimal time  which is not necessary the case 
if the control region  does not have this form.
For instance, if there is a ``cave'' in the control domain, as in Figure~\ref{fig:example.3}, 
the results of the present paper only ensure that the minimal time
is greater or equal than $a^2\!/2$ and smaller or equal than $b^2\!/2$.

The null-controllability in the critical time $T=a^2\!/2$ in Corollary~\ref{cor} also remains an open problem. 
Since for $\omega:=[(-b,-a)\cup(a,b)]\times(0,1)$ with $0<a<b\leq 1$
the minimal time is equal to $a^2\!/2$ and the Grushin equation is not 
null-controllable in this time~\cite{beauchard_2015}, 
we can conjecture that it is also the case in Corollary~\ref{cor}.

Finally, we have a positive result for the generalized Grushin equation $\partial_t -\partial_x^2 - q(x)^2\partial_y^2$ (Remark~\ref{rk:generalized_grushin}), but we lack a negative result corresponding to Theorem~\ref{th:main_negative}. If we had one, then we would also have a corollary similar to Corollary~\ref{cor}, with the minimal time being $q'(0)^{-1}\int_0^a q(s)\diff s$. 

\section{Proof of the positive result}\label{sec:pos}

This section is devoted to the proof of Theorem~\ref{th:main_postiive}.
To this end, we will adapt to our setting the fictitious control method 
which has already been used for partial differential equations for instance  in \cite{gonzalezperez2006}, \cite{coronlissy2014}, \cite{duprez2016positive} and \cite{ammar2007}. 

The strategy of the fictitious control method consists in 
building a solution of the control problem thanks to algebraic combinations of solutions to controlled problems.
We say the initial controls of the controlled problems are fictitious since they do not appear explicitly in the final equation.
The final controls will have less constraints than the initial ones,
they can have less components as in \cite{gonzalezperez2006,coronlissy2014,duprez2016positive} or a smaller support as in \cite{ammar2007} and in our case.

To build the control for the Grushin equation~\eqref{eq:control_problem_eng} thanks to the fictitious control method, we will use a previous result:

\begin{theorem}\label{th:equiv}
Assume that $\omega:=(-1,-a)\times(0,\pi)$ with $a\in(0,1)$. Then for each $T>a^2\!/2$, System~\eqref{eq:control_problem_eng} is null-controllable in time $T$.
\end{theorem}
The idea of the proof of Theorem~\ref{th:equiv} is the following:
The observability estimates obtained in \cite[Th. 1.4]{beauchard_2018} can be interpreted in terms of boundary controllability of System~\eqref{eq:control_problem_eng}
in $\Omega_a:=(-a,0)\times(0,\pi)$ and acting on the subset of the boundary $\{-a\}\times(0,\pi)\subset\partial\Omega_a$, which implies by a cutoff argument the internal controllability  of System~\eqref{eq:control_problem_eng} by acting on $(-1,-a+\varepsilon)\times(0,\pi)$ with $a\in(0,1)$:
A detailed proof of Theorem~\ref{th:equiv} is given in Appendix~\ref{sec:int_control}.

\begin{remark}
\begin{enumerate}
\item Thanks to \cite[Th. 1.4]{beauchard_2018}, we could also easily prove that the System~\eqref{eq:control_problem_eng} is not null-controllable on $(-1,-a)\times(0,\pi)$ in time $T<a^2\!/2$, but this is superseded by Theorem~\ref{th:main_negative} anyway. 

\item The equivalence between the boundary controllability and the internal
controllability is standard for the heat equation (see for instance
\cite[Theorem 2.2]{ammar2007}). However, the equivalence proved in the
reference is for $H^{1/2}$ boundary controls, but not in $L^2$. Moreover, the
Grushin equation is degenerate, so we can't directly apply the theorem of the reference. So we will prove Theorem~\ref{th:equiv} in Appendix~\ref{sec:int_control}.
\end{enumerate}
\end{remark}

\begin{proof}[Proof of Theorem~\ref{th:main_postiive}]
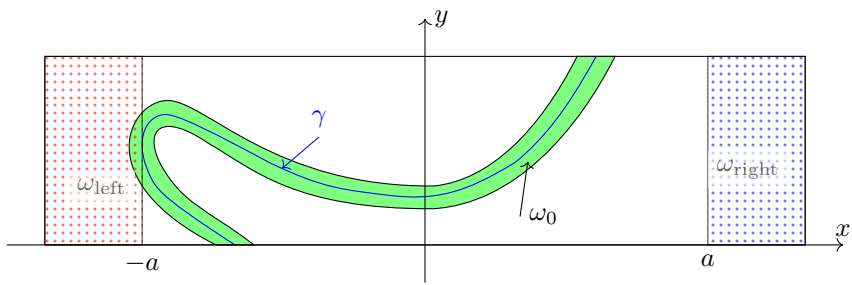
\begin{figure}
\centering
\input{omega0_left_right.tex}
\caption{Definition of $\omega_{\text{left}}$ (red) and $\omega_{\text{right}}$ (blue).}
\label{fig:omega 0 left right}
\end{figure}
Consider $\omega_0$ and $a$ given in~\eqref{omega_0} and~\eqref{def a}. Let $T>a^2\!/2$. We define

\begin{equation*}
\left\{\begin{array}{l}
\omega_{\text{left}}:=(-1,-a)\times(0,\pi),\\
\omega_{\text{right}}:=(a,1)\times(0,\pi).
\end{array}\right.
\end{equation*}
Such construction is illustrated in Figure~\ref{fig:omega 0 left right}.
Consider the following control problems
\begin{equation}\label{eq:f1}
\left\{\begin{array}{ll}
(\partial_t - \partial_x^2 - x^2\partial_y^2)f_{\text{left}}(t,x,y)  = \mathbf 1_{\omega_{\text{left}}} u_{\text{left}}(t,x,y) 
&  t\in [0,T], (x,y)\in \Omega\\
f_{\text{left}}(t,x,y) = 0 &  t \in [0,T], (x,y)\in \partial\Omega\\
f_{\text{left}}(0,x,y) = f_0(x,y),~f_{\text{left}}(T,x,y) = 0& (x,y)\in \Omega
\end{array}\right.
\end{equation}
and
\begin{equation}\label{eq:f2}
\left\{\begin{array}{ll}
(\partial_t - \partial_x^2 - x^2\partial_y^2)f_{\text{right}}(t,x,y)  = \mathbf 1_{\omega_{\text{right}}} u_{\text{right}}(t,x,y) 
&  t\in [0,T], (x,y)\in \Omega\\
f_{\text{right}}(t,x,y) = 0 &  t \in [0,T], (x,y)\in \partial\Omega\\
f_{\text{right}}(0,x,y) = f_0(x,y),~f_{\text{right}}(T,x,y) = 0& (x,y)\in \Omega.
\end{array}\right.
\end{equation}
Since $T>a^2\!/2$, according to Theorem~\ref{th:equiv}, null-controllability problems~\eqref{eq:f1} and~\eqref{eq:f2} admit solutions.
We now glue this two solutions $f_{\text{left}}$ and $f_{\text{right}}$ with an appropriate cutoff, given by Lemma~\ref{th:cutoff}:
\begin{lemma}\label{th:cutoff}
There exists a function $\theta\in C^\infty(\overline{\Omega})$ such that

\begin{equation*}
\left\{\begin{array}{l}
\theta(z) = 0\mbox{ for all }z\in\omega_{\text{\rm left}}\setminus\omega_0,\\
\theta(z) = 1\mbox{ for all }z\in\omega_{\text{\rm right}}\setminus\omega_0,\\
\supp(\nabla \theta) \subset \omega_0.
\end{array}\right.
\end{equation*}
\end{lemma}
By looking at Fig.~\ref{fig:omega 0 left right}, one should be convinced such a cutoff does exists; nevertheless, we provide a rigorous proof in Appendix~\ref{sec:cutoff_proof}. We define 
\[f:=\theta f_{\text{left}}+(1-\theta)f_{\text{right}}.\]

We remark that $f$ solves
\begin{equation*}
\left\{\begin{array}{ll}
(\partial_t - \partial_x^2 - x^2\partial_y^2)f(t,x,y)  =  u(t,x,y) 
&  t\in [0,T], (x,y)\in \Omega,\\
f(t,x,y) = 0 &  t \in [0,T], (x,y)\in \partial\Omega,\\
f(0,x,y) = f_0(x,y)& (x,y)\in \Omega,
\end{array}\right.
\end{equation*}
where
\[u:=\theta \mathbf 1_{\omega_{\text{left}}}u_{\text{left}}+(1-\theta)\mathbf 1_{\omega_{\text{right}}}u_{\text{right}}
+ (f_{\text{right}}-f_{\text{left}})(\partial_x^2 + x^2\partial_y^2)\theta
+2\partial_x(f_{\text{right}}-f_{\text{left}})\partial_x\theta
+2x^2\partial_y(f_{\text{right}}-f_{\text{left}})\partial_y\theta.\]
The properties of $\theta$ given in Lemma \ref{th:cutoff},
implies that $u$ is supported in $\omega_0\subset\omega$.
Using the fact that $f_{\text{left}}(T)=f_{\text{right}}(T)=0$, then $f(T)=0$.
Moreover, since $f_{\text{left}}$, $f_{\text{right}}\in L^2(0,T;V(\Omega))$ (see~\eqref{prop:well}), we deduce that $u\in L^2((0,T)\times\Omega)$.
\end{proof}

\section{Proof of the negative result}
\subsection{Proof of Theorem~\ref{th:main_negative}}\label{sec:neg}
We prove in this section the first non-null-controllability result 
(Theorem~\ref{th:main_negative}). We do this by adapting the method used by the second author to disprove the null-controllability when $\omega$ is the complement of a horizontal strip~\cite{koenig_2017}\footnote{What we are actually saying is that we recommend reading \cite[Section 2]{koenig_2017} to get a hang of the proof before reading the present Section~\ref{sec:neg}.}.

First, we note that according to the hypothesis that $\{(x,y_0),-a<x<a\}$ does
not intersect $\overline{\omega}$, for every $a'<a$, there exists a rectangle
of the form $\{-a'<x<a', |y-y_0|<\delta\}$ that does not intersect $\omega$
(see Fig.~\ref{fig:omega_neg}). Thus, we assume in the rest of this proof that
$\omega$ is the complement of this rectangle, i.e.

\begin{equation*}
\omega = \Omega\setminus\{-a'<x<a', |y-y_0|<\delta\}.
\end{equation*}
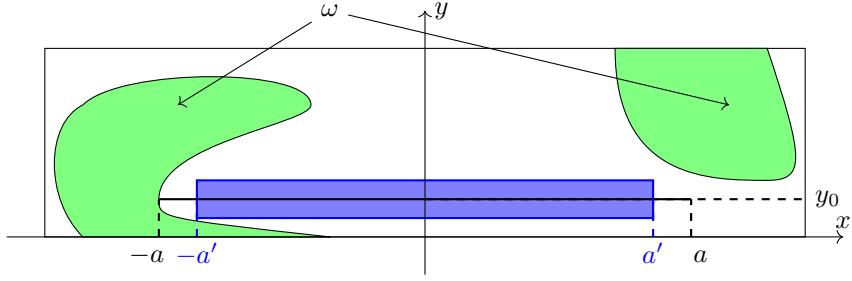
\begin{figure}
\centering
\input{omega_neg}
\caption{In green, the domain $\omega$. If we have a horizontal segment that does not touch $\omega$ (except maybe the extremities), we can find a
horizontal rectangle about the horizontal segment (in blue), with its length as
close as we want to the length of the horizontal segment. In the rest of the
proof, we will assume that $\omega$ is the complement of this blue rectangle.}
\label{fig:omega_neg}
\end{figure}
To disprove the null-controllability, we disprove the \emph{observability
inequality}, which is equivalent to the null-controllability (see Coron's
book~\cite[Theorem 2.44]{coron_2007} for a proof of this equivalence): there
exists $C>0$ 
such that for all $f_0$ in $L^2(\Omega)$, the solution $f$ to

\begin{equation*}
\left\{
\begin{array}{ll}
(\partial_t - \partial_x^2 - x^2\partial_y^2)f(t,x,y)  = 0 &  t\in [0,T], (x,y)\in \Omega,\\
f(t,x,y) = 0 &  t \in [0,T], (x,y)\in \partial\Omega,\\
f(0,x,y) = f_0(x,y)& (x,y)\in \Omega,
\end{array}
\right.
\end{equation*}
satisfies
\begin{equation}\label{eq:obs}
 \int_\Omega |f(T,x,y)|^2\diff x\diff y  \leq C \int_{[0,T]\times \omega}
|f(t,x,y)|^2\diff t \diff x \diff y.
\end{equation}

For each integer $n > 0$, let us note $v_n$ the first eigenfunction of $-\partial_x^2 + (n x)^2$ with Dirichlet boundary conditions on $(-1,1)$, and with associated eigenvalue $\lambda_n$. Then the function $v_n(x)\sin(ny)$ is an eigenfunction of $-\partial_x^2 - x^2\partial_y^2$ with Dirichlet boundary condition on $\partial\Omega$, and with eigenvalue $\lambda_n$. We will disprove the observability inequality~\eqref{eq:obs} with solutions $f$ of the form

\begin{equation*}
f(t,x,y) = \sum_{n>0} a_n v_n(x) \eu^{-\lambda_n t} \sin(ny).
\end{equation*}

To avoid irrelevant summability issues, we will assume that all the sums are finite. The observability inequality~\eqref{eq:obs} applied to these functions reads:

\begin{equation}\label{eq:obs2}
|f(T,\cdot,\cdot)|_{L^2(\Omega)}^2= \frac\pi2
\sum_{n>0} |a_n|^2|v_n|_{L^2}^2 \eu^{-2\lambda_n T}
\leq C \int_{[0,T]\times \omega} \left\lvert \sum_{n>0} a_n v_n(x) \eu^{-\lambda_n t}\sin(ny) \right\rvert^2\diff t\diff x\diff y =
C|f|_{L^2([0,T]\times\omega)}^2.
\end{equation}
Note that we have not supposed the $v_n$ to be normalized in $L^2$. Instead, we will 
find more convenient to normalize them by the condition $v_n(0) = 1$. Let us remind that thanks to Sturm-Liouville theory, $v_n$ is even and $v_n(0)\neq 0$ (see \cite[Lemma 2]{beauchard_2014}).
\begin{figure}
\centering
\input{U}
\caption{The domain $U$.}
\label{fig:U}
\end{figure}
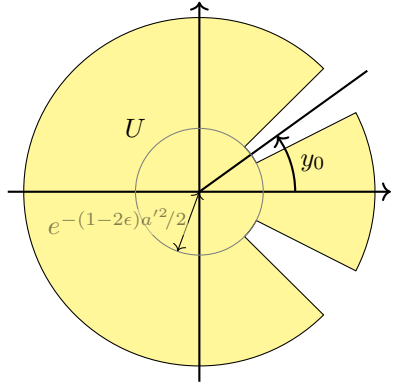%

The main idea is to relate this inequality to an estimate on entire polynomials. Let 
 $\varepsilon\in(0,1/2)$ be a small real number to be chosen later (it will depend only on 
$T$ and $a'$). Then we have Lemma~\ref{th:lemma_hol}:
\begin{lemma}\label{th:lemma_hol}
Let\footnote{We note $D(a,r) = \{z\in \set C, |z-a|<r\}$ the complex open disk of center $a$ and radius $r$.} $U = \{|z|<1, ||\arg(z)|-y_0| > \delta/2\}\cup 
D(0,\eu^{-(1-2\varepsilon)a'^2\!/2})$ (see Fig.~\ref{fig:U}).
The inequality~\eqref{eq:obs2} implies that there exists $N>0$ and $C>0$ such that 
for all entire polynomials of the form $p(z) = \sum_{n>N} a_n z^n$,

\begin{equation}\label{eq:obs_lemma}
|p|_{L^2(D(0,e^{-T}))} \leq C |p|_{L^\infty(U)}.
\end{equation}
\end{lemma}
\begin{proof}
About the left-hand side of the observability inequality~\eqref{eq:obs2}, we first note that by writing the integral on a disk $D=D(0,r)$ of $z^n\overline{z}^m$ in polar coordinates, we find that the functions $z\mapsto z^n$ are orthogonal on $D(0,r)$. So, we have for all polynomials\footnote{We denote $\lambda$ the Lebesgue measure on $\set C\simeq \set R^2$. I.e. for a function $f\colon \set C\to\set C$, if $(x,y)\in \set R^2\mapsto f(x+\iu y)$ is integrable, then $\int_{\set C} f(z)\diff\lambda(z) = \int_{\set R^2} f(x+iy)\diff x\diff y$.} $\sum_{n> 0} a_n z^{n-1}$:

\[\int_{D(0,\eu^{-T})} \Big\lvert\sum_{n>0} a_n z^{n-1}\Big\rvert^2 \diff \lambda(z) = \sum_{n>0} |a_n|^2 \int_{D(0,\eu^{-T})} |z|^{2n-2} \diff \lambda(z)\]
and computing $\int_{D(0,\eu^{-T})}|z|^{2n-2}\diff \lambda(z) $ in polar coordinates:

\begin{equation*}
\int_{D(0,\eu^{-T})} \Big\lvert\sum_{n>0} a_n z^{n-1}\Big\rvert^2 \diff \lambda(z)= \sum_{n>0}\frac\pi n |a_n|^2 \eu^{-2nT}.
\end{equation*}
Moreover, we know from basic spectral analysis that writing $\lambda_n = n + \rho_n$, $(\rho_n)_{n\geq 0}$ is bounded (see~\cite[Section 3.3]{beauchard_2014}) and that $|v_n|_{L^2(-1,1)}^2\geq c|n|^{-1/2}$ (see for instance~\cite[Lemma 21]{koenig_2017}), so that

\begin{align*}
\int_{D(0,\eu^{-T})} \Big\lvert\sum_{n>0} a_n z^{n-1}\Big\rvert^2 \diff \lambda(z)&\leq \pi c^{-1}\sum_{n>0}|a_n|^2|v_n|^2_{L^2} \eu^{-2nT}\\
&\leq \pi c^{-1}\eu^{2\sup_k\rho_k T} \sum_{n>0} |a_n|^2 |v_n|^2_{L^2}\eu^{-2\lambda_n T}\\
&=C|f(T,\cdot,\cdot)|^2_{L^2(\Omega)}.
\end{align*}
Thus, the observability inequality~\eqref{eq:obs2} implies for another constant $C$ (that depends on time $T$ but it does not matter):

\begin{equation}\label{eq:obs3.5}
\int_{D(0,\eu^{-T})} \Big\lvert\sum_{n>0} a_n z^{n-1}\Big\rvert^2 \diff \lambda(z) \leq C|f|_{L^2([0,T]\times \omega)}^2.
\end{equation}

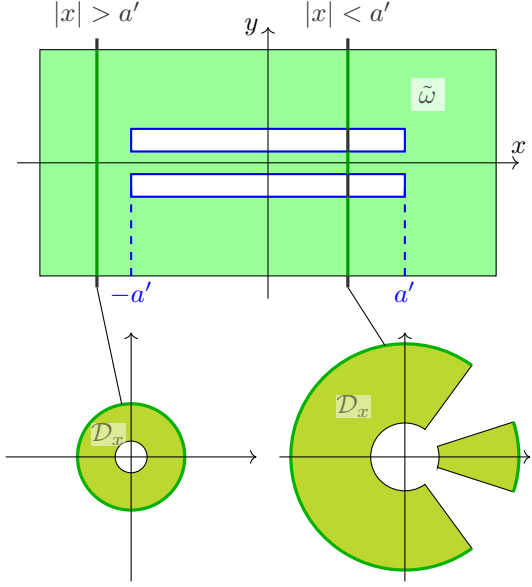
\begin{figure}
 \begin{minipage}[c]{0.6\textwidth}
\centering
\input{omega_symetrique}
 \end{minipage}\hfill%
 \begin{minipage}[c]{0.4\textwidth}
\caption{Above: The domain $\tilde \omega$ is the union of $\omega$ and of its symmetric with respect to the axis $x=0$. Below: the domain $\mathcal D_x$ is defined to be the set of complex numbers $z$ of modulus between $\eu^{-T-(1-\varepsilon)x^2\!/2}$ and $\eu^{-(1-\varepsilon)x^2\!/2}$, and with argument such that $(x,\arg(z))\in\tilde \omega$. It is a partial ring if $|x|<a'$ and a whole ring if $|x|>a'$. Indeed, if we take a slice of $\tilde \omega$ by fixing $x$, when $|x|<a'$, we don't have the whole interval $(-\pi,\pi)$, but when $|x|>a'$, the slice is the whole interval $(-\pi,\pi)$.}
\label{fig:omega_symetric}
\end{minipage}
\end{figure}%

To bound the right-hand side, we begin by writing $\sin(ny) = (\eu^{\iu ny}-e^{-\iu ny})/2\iu$, so that the right-hand side satisfies

\begin{equation*}
|f|_{L^2([0,T]\times\omega)}^2  \leq \frac12 \int_{[0,T]\times \omega} 
\left(\Big\lvert\sum_n a_n v_n(x)\eu^{\iu ny-\lambda_n t}\Big\rvert^2 + \Big\lvert\sum_n a_n 
v_n(x)\eu^{-\iu ny-\lambda_n t}\Big\rvert^2\right)\diff t\diff x\diff y.
\end{equation*}
Then, noting $\tilde \omega = \omega\cup \{(x,-y),(x,y)\in\omega\}$ (see Fig.~\ref{fig:omega_symetric}), we rewrite this 
as 

\begin{equation*}
|f|_{L^2([0,T]\times\omega)}^2 \leq \frac12 \int_{[0,T]\times \tilde\omega} 
\Big\lvert\sum_n a_n v_n(x)\eu^{\iu ny-\lambda_n t}\Big\rvert^2 \diff t\diff x\diff y.
\end{equation*}

Then, we again write $\lambda_n = n + \rho_n$ and $v_n(x) = \eu^{-(1-\varepsilon)nx^2\!/2}w_n(x)$, so that with $z_x(t,y) = \eu^{-t+\iu y-(1-\varepsilon)x^2\!/2}$, the previous inequality implies:

\begin{equation}\label{eq:obs3.6}
|f|_{L^2([0,T]\times \omega)}^2 \leq 
\frac12\int_{[0,T]\times \tilde \omega} \Big\lvert\sum_{n>0} a_n w_n(x)\eu^{-\rho_n t} 
z_x(t,y)^n\Big\rvert^2\diff t\diff x\diff y.
\end{equation}

Now for each $x\in(-1,1)$, we make the change of variables $z_x = 
\eu^{-t+\iu y-(1-\varepsilon)x^2\!/2}$, for which $\diff t\diff y = |z_x|^{-2}\diff\lambda(z)$. For each 
$x$, let us note $\mathcal D_x$ the image of the set we integrate on for this change of variables (see Fig.~\ref{fig:omega_symetric}), that is,
\begin{equation*}
\mathcal D_x = \{\eu^{-(1-\varepsilon)x^2\!/2}\eu^{-t+\iu y}, 0<t<T, (x,y)\in\tilde\omega\}.
\end{equation*}

We get
\begin{equation*}
\int_{[0,T]\times \tilde \omega} \Big\lvert\sum_{n>0} a_n w_n(x)\eu^{-\rho_n t} z_x(t,y)^n\Big\rvert^2\diff t\diff x\diff y = \int_{-1}^1\int_{\mathcal D_x}\Big\lvert\sum_{n>0} a_n w_n(x)e^{-\rho_n t} z^{n-1}\Big\rvert^2\diff \lambda(z)\diff x,
\end{equation*}
where we kept the notation $e^{-\rho_n t}$ for simplicity instead of expressing it as a function of $z$ and $x$ (we have $e^{-\rho_n t} = |e^{(1-\varepsilon)x^2\!/2}z|^{\rho_n}$). With this change of variables, the inequality~\eqref{eq:obs3.6} becomes

\begin{equation}\label{eq:obs4}
|f|_{L^2([0,T]\times \omega)}^2 \leq 
\frac12\int_{-1}^1\int_{\mathcal D_x}\Big\lvert\sum_{n>0} a_n w_n(x)\eu^{-\rho_n t}
z^{n-1}\Big\rvert^2\diff \lambda(z)\diff x. 
\end{equation}

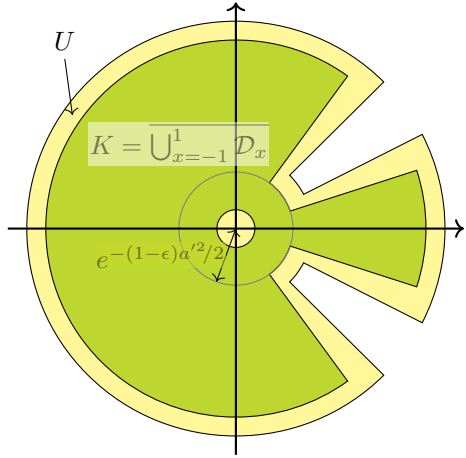
\begin{figure}
 \begin{minipage}{0.6\textwidth}
\centering
\input{D}
 \end{minipage}\hfill%
 \begin{minipage}[c]{0.4\textwidth}
\caption{The domain $K$, in green, is (the closure of) the union of 
the domains $\mathcal D_x$ described in Fig.~\ref{fig:omega_symetric}. For the domain 
$K$ the radius of the inner part-of-circle is the largest radius of the $\mathcal 
D_x$ that is a full ring, i.e. $\eu^{-(1-\varepsilon)a'^2\!/2}$. We also show $U$ for 
comparison. Notice that $U$ has been defined to be a neighborhood of $K$ that is 
star-shaped with respect to $0$.}
\label{fig:D}
\end{minipage}
\end{figure}

We want to bound the right-hand side by $\lvert \sum_{n>0} a_n 
z^{n-1}\rvert_{L^\infty(U)}^2$. To do this, we use the following Lemma~\ref{th:est_hol}, that we prove in Section~\ref{sec:lemma_hol_proof}. This Lemma is a rigorous statement of the fact that $\rho_n$ and $w_n$ are small.\puncfootnote{At least small enough for our purposes, the good notion is that of Symbols, see Definition~\ref{def:symbol} below.}
\begin{lemma}\label{th:est_hol}
 Let \(K\noic\) be a compact subset of $\set C$ and let \(V\noic\) be a bounded neighborhood of \(K\noic\) that 
is star-shaped with respect to $0$. Then, there exists $C>0$ and $N>0$ such that for 
every $x\in(-1,1)$, for every $\tau \in [0,T]$, and for every polynomial 
$\sum_{n>N} a_n z^{n-1}$:

\begin{equation*}\label{eq:est_hol}
 \Big\lvert \sum_{\mathclap{n>N}} a_n w_n(x) \eu^{-\rho_n\tau}z^{n-1}\Big\rvert_{{L^\infty(K)}} 
\leq C\Big\lvert \sum_{\mathclap{n>N}} a_n z^{n-1}\Big\lvert_{L^\infty(V)}.
\end{equation*}
\end{lemma}

From now on, we assume that $a_n = 0$ for $n\leq N$. In Lemma~\ref{th:est_hol}, we choose $K = \overline{\bigcup_{x=-1}^1 \mathcal D_x}$ and $V = U$, where $U$ was defined in the statement of Lemma~\ref{th:lemma_hol} (see 
Fig.~\ref{fig:D}). Notice that by definition of $\mathcal D_x$, $K$ is the union of the ring $\{e^{-T-(1-\varepsilon)/2} \leq |z| \leq e^{-(1-\varepsilon)a'^2\!/2}\}$ and of the partial ring $\{e^{-T-(1-\varepsilon)a'^2\!/2}\leq|z|\leq1, ||\arg(z)|-y_0|\geq\delta\}$, and so $U$ is a neighborhood of $K$ that is star-shaped with respect to $0$, and the hypotheses of Lemma~\ref{th:est_hol} are satisfied. We get, by taking $\tau = t$ in Lemma~\ref{th:est_hol} (let us remind that $t$ is a function of $z$ and $x$):

\begin{equation*}
 \Big\lvert \sum_{\mathclap{n>N}} a_n w_n(x) 
\eu^{-\rho_n t}z^{n-1}\Big\rvert_{L^\infty(K)} 
\leq C\Big\lvert \sum_{\mathclap{n>N}} a_n z^{n-1}\Big\lvert_{L^\infty(U)}.
\end{equation*}

Since $K$ contains every $\mathcal D_x$, we have for every $x\in(-1,1)$:
\begin{equation}\label{eq:obs_n}
\int_{\mathcal D_x}\Big\lvert\sum_{n>N} a_n w_n(x)\eu^{-\rho_n t}
z^{n-1}\Big\rvert^2\diff \lambda(z)
\leq \lambda(K)\Big\lvert \sum_{\mathclap{n>N}} a_n w_n(x) 
\eu^{-\rho_n t}z^{n-1}\Big\rvert_{L^\infty(K)}^2 \leq \pi C^2\Big\lvert \sum_{\mathclap{n>N}} a_n z^{n-1}\Big\lvert_{L^\infty(U)}^2.
\end{equation}

Then, we plug this estimate~\eqref{eq:obs_n} into the estimate~\eqref{eq:obs4} to get for some constant $C>0$:
\begin{equation}
|f|_{L^2([0,T]\times \omega)}^2
\leq \pi C^2\Big\lvert\sum_{\mathclap{n>N}} a_n z^{n-1}\Big\rvert^2_{L^\infty(U)}.
\end{equation}

Combining this with the estimate~\eqref{eq:obs3.5} on the left-hand side, we get for some constant $C>0$:

\begin{equation*}
\int_{D(0,\eu^{-T})} \Big\lvert\sum_{\mathclap{n>N}} a_n z^{n-1}\Big\rvert^2 \diff \lambda(z) \leq C \Big\lvert\sum_{\mathclap{n>N}} a_n z^{n-1}\Big\rvert^2_{L^\infty(U)},
\end{equation*}
which is, up to a change of summation index $n' = n-1$, the estimate~\eqref{eq:obs_lemma} we wanted.
\end{proof}

\begin{figure}
 \begin{minipage}[c]{0.5\textwidth}
\centering
\input{runge}
 \end{minipage}\hfill%
 \begin{minipage}[c]{0.5\textwidth}
\caption{When the disk $D(0,\eu^{-T})$ (in red) is not included in $U$, we can find holomorphic functions that are small in $U$ but arbitrarily big in $D(0,\eu^{-T})$. For instance, we can construct with Runge's theorem a sequence of polynomials that converges to $z\mapsto(z-z_0)^{-1}$ away from the blue line.}
\label{fig:runge}
\end{minipage}
\end{figure}
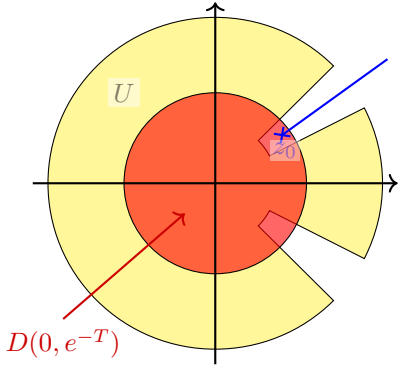

Let us assume $T<(1-2\varepsilon)a'^2\!/2$. We want to disprove the inequality~\eqref{eq:obs_lemma} of Lemma~\ref{th:lemma_hol}. We will use Runge's theorem (see for instance Rudin's texbook \cite[Theorem 13.9]{rudin_1986}) to construct a counterexample.
\begin{prop}[Runge's theorem]
  Let \(U\noic\) be a connected and simply connected open subset of $\set C$, and let $f$ be a holomorphic function on \(U\noic\). There exists a sequence $(p_k)$ of polynomials that converges uniformly on every compact subsets of \(U\noic\) to $f$.
\end{prop}

By definition of $U$, there exists a complex number $z_0\in D(0,e^{-T})$ which is non-adherent to $U$ (see Fig.~\ref{fig:runge}). Then, according to Runge's theorem, there exists a sequence of polynomials $\tilde p_k$ that converges uniformly on every compact subset of $\set C\setminus (z_0[1,+\infty))$ to $z\mapsto (z-z_0)^{-1}$, and we define $p_k(z) = z^{N+1}\tilde p_k(z)$. Then, the family $p_k$ is a counter example to the inequality on entire polynomials~\eqref{eq:obs_lemma}. Indeed, since $z^{N+1}(z-z_0)^{-1}$ is bounded on $U$, $p_k$ is uniformly bounded on $U$, thus, the right-hand side of the inequality~\eqref{eq:obs_lemma} is bounded. But since $z_0$ is in $D(0,e^{-T})$, $z^{N+1}(z-z_0)^{-1}$ has infinite $L^2$-norm in $D(0,e^{-T})$, and thanks to Fatou's Lemma, $|p_k|_{L^2(D(0,e^{-T}))}$ tends to $+\infty$ as $k\to+\infty$.

Thus, the Grushin equation is not null-controllable if $T<(1-2\varepsilon)a'^2\!/2$. But $a'$ can be chosen arbitrarily close to $a$, and $\varepsilon$ arbitrarily small, so the Grushin equation is not null-controllable if $T<a^2\!/2$.\hfill \qed

\subsection{Proof of Lemma \ref{th:est_hol}}\label{sec:lemma_hol_proof}
For the proof of this Lemma, we will need a few definition and theorems.
\begin{definition}[Definition 9 of~\cite{koenig_2017}]\label{def:symbol}
 Let $r$ be a non-decreasing function $r\colon(0,\pi/2)\to \set R$. For $0<\theta<\pi/2$, we note $\Delta_\theta=\{z\in \set C, |z|>r(\theta), \lvert\arg(z)\rvert<\theta\}$. We define $\mathcal S_r$ 
the set of functions $\gamma$ from $\bigcup_{0<\theta<\pi/2}\Delta_\theta$ to $\set C$, that are holomorphic and with subexponential growth on each $\Delta_\theta$, i.e. such that for all $0<\theta<\pi/2$ and $\varepsilon>0$,
\begin{equation}
 p_{\theta,\varepsilon}(\gamma) \coloneqq \sup_{z\in \Delta_\theta} |\gamma(z)|e^{-\varepsilon |z|} 
< +\infty.
\end{equation}
We endow $\mathcal S_r$ with the topology of the seminorms $(p_{\theta,\varepsilon})_{0<\theta<\pi/2, \varepsilon>0}$. We will call elements of $\mathcal S_r$ \emph{symbols}.
\end{definition}

\begin{theorem}[Theorem 18 of\footnote{In the reference \cite{koenig_2017}, the Theorem is stated sightly differently, but the two formulations are equivalent.} \cite{koenig_2017}]\label{th:default}
Let $r\colon(0,\pi/2)\to \set R$ be a non-decreasing function and a symbol $\gamma\in \mathcal S_r$ and $N = \lfloor \inf_{\theta}(r(\theta)\rfloor$. Let $H_\gamma$ be the operator on polynomials with the first $N$ derivatives vanishing at $0$, defined by:

\begin{equation*}
 H_\gamma \colon\sum_{n>N} a_n z^n\longmapsto \sum_{n>N} \gamma(n)a_n z^n.
\end{equation*}
The map $\gamma \in \mathcal S_r \mapsto H_\gamma$ satisfies the following continuity-like estimate: for each compact $K\noic$ and each neighborhood $V\noic$ of $K\noic$ that is star-shaped with respect to $0$, there exists a constant $C>0$ and a finite number of semi\-norms $(p_{\theta_i,\varepsilon_i})_{1\leq i\leq n}$ of 
$\mathcal S_r$ such that for every symbol $\gamma$ and polynomial of the form $f = \sum_{n>N} a_n z^n$:
\begin{equation}\label{eq:est_default}
 |H_\gamma(f)|_{L^\infty(K)} \leq C\sup_{1\leq i\leq n} 
p_{\theta_i,\varepsilon_i}(\gamma) |f|_{L^\infty(V)}.
\end{equation}
\end{theorem}

\begin{theorem}[Theorem 22 and Proposition 25 of \cite{koenig_2017}]\label{th:symbols}
We remind that $\lambda_n$ is the first eigenvalue of $-\partial_x^2 +(nx)^2$ with Dirichlet boundary conditions on $(-1,1)$, $v_n$ is the associated eigenfunction that satisfies $v_n(0) = 1$ and the definition of $w_n = \eu^{(1-\varepsilon)nx^2\!/2}v_n(x)$.

There exists a non-decreasing function\footnote{In the reference \cite{koenig_2017}, it is not clear we can choose the same $r$ for (i) and (ii): we have $r_1$ such that (i) holds and $r_2$ such that (ii) holds. Then, we just have to choose $r = \max(r_1,r_2)$.} $r\colon(0,\pi/2)\to \set R$ such that for $N = \lfloor \inf_\theta r(\theta)\rfloor$:
\begin{enumerate}
\item there exists a symbol $\gamma \in\mathcal S_r$ such that for $n>N$, $\lambda_n = n + \gamma(n)\eu^{-n}$;
\item for each $x\in(-1,1)$, there exists a symbol $w(x)\in \mathcal S_r$ such that for $n>N$, $w_n(x) = w(x)(n)$, and moreover, the family $(w(x))_{-1<x<1}$ is a bounded family of $\mathcal S_r$.
\end{enumerate}
\end{theorem}

\begin{proof}[Proof of Lemma \ref{th:est_hol}]
We want to bound $\sum_{n>N} w_n(x)\eu^{-\rho_n\tau} a_n z^{n-1}$ by $\sum_{n>N} a_n z^{n-1}$. To do this, we prove that $w_n(x)\eu^{-\rho_n\tau}$ is a symbol in the sense of Definition~\ref{def:symbol}, and then apply the Theorem~\ref{th:default}.

Let $\gamma\in \mathcal S_r$ and $w(x)\in\mathcal S_r$ obtained by Theorem~\ref{th:symbols} and let $\rho(\alpha)\coloneqq \gamma(\alpha)\eu^{-\alpha}$. Also let $N = \lfloor\inf_{\theta}r(\theta)\rfloor$. The function $\rho$ satisfies for every $n>N$, $\rho(n) = \rho_n$. Finally, for 
$0< \tau < T$ and $x\in(-1,1)$, let $\gamma_{\tau,x}$ defined by:
\[\gamma_{\tau,x}(\alpha) = w(x)(\alpha+1)\eu^{-\rho(\alpha+1)\tau},\]
so that:
\begin{equation}\label{eq:lemma}\sum_{n> N+1} w_n(x)\eu^{-\rho_n\tau} a_n z^{n-1} = H_{\gamma_{\tau,x}}\left(\sum_{n>N+1} a_n z^{n-1}\right).
\end{equation}
Note that we evaluated $w(x)$ and $\rho$ in $\alpha+1$ instead of $\alpha$ because we want to multiply $z^{n-1}$ by $w_n(x)\eu^{-\rho_n\tau}$. This is not a problem because the domain of definition of a symbol is invariant by $z\mapsto z+1$, and thus if $\gamma\in\mathcal S_r$ is a symbol, then so is $\tilde\gamma = \gamma(\cdot+1)$, and we moreover have $p_{\theta,\varepsilon}(\tilde\gamma) \leq p_{\theta,\varepsilon}(\gamma)$.

We then show that the family $(\gamma_{\tau,x})_{0<\tau<T, x\in(-1,1)}$ is in $\mathcal S_r$, and is bounded. Since $\rho(\alpha) = e^{-\alpha} \gamma(\alpha)$ with $\gamma$ having sub-exponential growth (by definition of $\mathcal S_r$), $|\rho(\alpha)|$ is bounded on every $\Delta_{\theta}$ by some $c_\theta$. So, we have for $0<\tau <T$ and $\alpha\in \Delta_{\theta}$:

\begin{equation*}
 |\eu^{-\rho(\alpha)\tau}| \leq\eu^{|\rho(\alpha)\tau|}\leq \eu^{T c_\theta}.
\end{equation*}
So $\eu^{-\rho(\alpha)\tau}$ is bounded for $\alpha \in \Delta_{\theta}$, and in particular has sub-exponential growth. Since $\rho$ is holomorphic, so is $\alpha\mapsto \eu^{-\rho(\alpha)\tau}$, thus, $\alpha\mapsto \eu^{-\rho(\alpha)\tau}$ is in $\mathcal S_r$. Moreover, the bound $|\eu^{-\rho(\alpha)\tau}| \leq e^{Tc_\theta}$ is uniform in $0<\tau <T$, so $(\eu^{-\rho\tau})_{0<\tau<T}$ is a bounded family of 
$\mathcal S_r$.

Moreover, we already know that $(w(x))_{x\in(-1,1)}$ is a bounded family in $\mathcal S_r$. Since $\gamma_{\tau,x}$ is the multiplication of $w(x)$ and $\eu^{-\rho \tau}$ and since the multiplication is continuous\footnote{See for instance~\cite[Proposition~12]{koenig_2017}, but it is elementary.} in $\mathcal S_r$, $(\gamma_{\tau,x})_{0<\tau<T, x\in(-1,1)}$ is a bounded family of $\mathcal S_r$.

So according to the estimate~\eqref{eq:est_default} of Theorem~\ref{th:default}, if $V$ is a bounded neighborhood of $K$ that is star-shaped with respect to $0$, there exists $C>0$ independent of $\zeta,x$, such that:

\[\bigg|H_{\gamma_{\tau,x}}\bigg(\sum_{n> N+1} a_n z^{n-1}\bigg)\bigg|_{L^\infty(K)}\mkern-10mu 
\leq 
C\bigg|\sum_{n> N+1} a_nz^{n-1}\bigg|_{L^\infty(V)}.\]

So, thanks to equation \eqref{eq:lemma}:

\begin{equation*}
\bigg|\sum_{n> N} w_n(x) a_n z^{n-1}\eu^{-\rho_n\tau}\bigg|_{L^\infty(K)}\mkern-15mu
\leq  C\bigg|\sum_{n> N} a_nz^{n-1}\bigg|_{L^\infty(V)}\mkern-5mu.\qedhere
\end{equation*}
\end{proof}

\section{Computation of the minimal time for some control domains}\label{sec:cor}
\begin{figure}
\centering
\input{omega_corollary}
\caption{In green, the domain $\omega$. At $y=y_0$, the function $\max(\gamma_2^-,\gamma_1^+)$ takes its maximum $a$. Then, the open interval $\{(x,y_0),-a<x<a\}$ is disjoint from $\overline{\omega}$. So, the Grushin equation is not null-controllable in time $T<a^2\!/2$. Also, if we take a path $\gamma$ (here in blue) from the bottom boundary to the top boundary that is close to the boundary of $\omega$ around $y = y_0$, then, we can apply Theorem~\ref{th:main_postiive}, and the Grushin equation is null-controllable in time $T>a^2\!/2$.}
\label{fig:omega_corollary}
\end{figure}
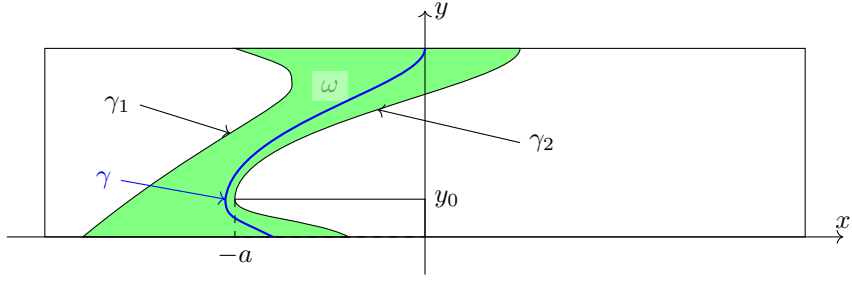

In this section, we prove the Corollary~\ref{cor}. By looking at Fig.~\ref{fig:omega_corollary}, one can be convinced Theorems~\ref{th:main_postiive} and~\ref{th:main_negative} will give $a^2\!/2$ as a minimal time of null-controllability, but let us actually prove it. Let us recall that $\omega= \{(x,y)\in\Omega,\gamma_1(y)<x<\gamma_2(y)\}$ and that $a = \max(\max(\gamma_2^-),\max(\gamma_1^+))$.

\noindent\textit{First step: lower bound of the minimal time.} For this step, we only have to treat the case $a>0$. By definition of $a$, for every $\varepsilon>0$, there exists $y_\varepsilon\in(0,\pi)$ such that $\gamma_2(y_\varepsilon)<-a+\varepsilon$ or $\gamma_1(y_\varepsilon)>a-\varepsilon$. Then, since $\gamma_1<\gamma_2$, the segment $\{(x,y_\varepsilon),|x|<a-\varepsilon\}$ is dijoint from $\overline\omega$, and thanks to Theorem~\ref{th:main_negative}, the Grushin equation~\eqref{eq:control_problem_eng} is not null-controllable on $\omega$ in time $T<(a-\varepsilon)^2\!/2$. This is true for every $\varepsilon>0$.

\noindent\textit{Second step: upper bound of the minimal time.} Let $\varepsilon>0$ small enough so that $\gamma_2-\gamma_1 > \varepsilon$. Let $\tilde\gamma_1 =\max(\gamma_1, -a-\varepsilon)$ and $\tilde\gamma_2 =\min(\gamma_2, a+\varepsilon)$. By using the information $\gamma_2-\gamma_1>\varepsilon$, $\gamma_2\geq -a$, $\gamma_1\leq a$ and by looking at the different cases, we readily get $\tilde \gamma_2 - \tilde \gamma_1 \geq \varepsilon$. Then, we define the path 

\[\gamma\colon s\in[0,\pi] \mapsto \left(\frac{\tilde\gamma_1(s) + \tilde\gamma_2(s)}2, s\right).\]
This path goes from the bottom boundary to the top boundary, and satisfies $\lvert\abscissa(\gamma)\rvert\leq a+\varepsilon$ and $\gamma_1 +\varepsilon/2 \leq \abscissa(\gamma) \leq \gamma_2-\varepsilon/2$. Therefore, for $\eta>0$ small enough, we have

\[\omega_0 \coloneqq \{z\in\Omega, \distance(z,\,\ran(\gamma))<\eta\}\subset \omega.\]
Thus, the Theorem~\ref{th:main_postiive} tells us the Grushin equation~\eqref{eq:control_problem_eng} is null-controllable in time $T>(a+\varepsilon)^2\!/2$. This is true for every $\varepsilon>0$.\hfill \qed

\begin{appendix}
\section{Proof of Theorem \ref{th:equiv}}\label{sec:int_control}
In this section, we will show that the boundary null-controllability of System~\eqref{eq:control_problem_eng}
(this notion is recalled in Definition~\ref{def:contr bord}) 
implies the internal null-controllability 
of System~\eqref{eq:control_problem_eng}. 
The argument is standard, but for the sake of clarity, we include it in the present paper.

Let $\Omega_L:=(-L,1)\times(0,\pi)$ and $\Gamma_L:=\{-L\}\times(0,\pi)$ with $L\in(0,1)$.
Consider the system
\begin{equation}\label{eq:control bord}
\left\{
\begin{array}{ll}
(\partial_t - \partial_x^2 - x^2\partial_y^2)f(t,x,y)  = 0 &  t\in [0,T], (x,y)\in \Omega_L,\\
f(t,x,y) = 0 &  t \in [0,T], (x,y)\in \partial\Omega_L\backslash\Gamma_L,\\
f(t,x,y) = v(t,y) &  t \in [0,T], (x,y)\in \Gamma_L,\\
f(0,x,y) = f_0(x,y)& (x,y)\in \Omega_L,
\end{array}
\right.
\end{equation}
where $f_0\in L^2(\Omega_L)$ is the initial data and $v\in L^2((0,T)\times\Gamma_L)$ is the control.
The solution to System~\eqref{eq:control bord} will be considered in the following sense:

\begin{definition}\label{def:trans}
Let $f_0\in V(\Omega_L)'$ and $v\in L^2((0,T)\times\Gamma_L)$ be given. It  will be said that $y\in L^2((0,T)\times\Omega_L)$ 
is a \emph{solution by transposition} to System~\eqref{eq:control bord}, if 
for each $g\in L^2((0,T)\times\Omega_L)$, we have:
\begin{equation*}\int_{(0,T)\times\Omega_L}f(t,x,y)g(t,x,y)\diff x\diff y\diff t =\langle f_0,\varphi(0)\rangle_{V(\Omega_L)',V(\Omega_L)}
+\int_0^T\int_0^1v(t,y)\varphi_x(t,-L,y)\diff y\diff t,\end{equation*}
where $\varphi$ is the solution to
\begin{equation}\label{eq:dual}
\left\{
\begin{array}{ll}
(-\partial_t - \partial_x^2 - x^2\partial_y^2)\varphi(t,x,y)  = g(t,x,y) &  t\in [0,T], (x,y)\in \Omega_L,\\
\varphi(t,x,y) = 0 &  t \in [0,T], (x,y)\in \partial\Omega_L,\\
\varphi(T,x,y) = 0& (x,y)\in \Omega_L.
\end{array}
\right.
\end{equation}
\end{definition}
Using standard argument of the semi-group theory, 
one can prove the well posedness of System~\eqref{eq:control bord}
(see for instance 
\cite[Prop. 2.2]{Fernandez2010})
in the case of the heat equation):
\begin{prop}
For all $f_0\in V(\Omega_L)'$ and $v\in L^2((0,T)\times\Gamma_L)$,
there exists a unique solution by transposition 
to System~\eqref{eq:control bord}.
\end{prop}

We now recall the notion of null-controllability of System~\eqref{eq:control bord}
(or boundary null-controllability of System~\eqref{eq:control_problem_eng}):
\begin{definition}\label{def:contr bord}
System~\eqref{eq:control bord} is said to be \emph{null-controllable} in time $T$ if for each 
initial data $f_0\in L^2(\Omega_L)$
there exists a control $v\in L^2((0,T)\times\Gamma_L)$ such that $f(T,\cdot,\cdot)=0$
in $V(\Omega_L)'$.
\end{definition}

The observability estimates obtained in \cite[Th. 1.4]{beauchard_2018}  can be interpreted in terms of controllability 
as follows:
\begin{theorem}\label{th:bord}
One has
\begin{enumerate}
\item For each $T>L^2\!/2$, System~\eqref{eq:control bord} is null-controllable in time $T$.
\item For each $T<L^2\!/2$, System~\eqref{eq:control bord} is not null-controllable in time $T$.
\end{enumerate}
\end{theorem}

We will now prove Theorem~\ref{th:equiv}:

\begin{proof}[Proof of Theorem~\ref{th:equiv}]
Assume that  $T:=(a+\varepsilon)^2\!/2$ with $\varepsilon\in (0,1-a)$ and let 
$f_0\in L^2(\Omega)$.
Using Item (i) of Theorem~\ref{th:bord} for $L:=a+\varepsilon/2$,
consider the controlled solution  $f_{\text{boundary}}$ to System~\eqref{eq:control bord} in time $T$ for the initial data $f_{0|\Omega_L}$ 
with a control $v\in L^2((0,T)\times\Gamma_L)$.
Denote by $f_{\text{free}}$ the uncontrolled solution to System~\eqref{eq:control_problem_eng} for $u:=0$. 
We define 

\begin{equation*}
f_{\text{internal}}:=\eta\theta_1 f_{\text{free}}+(1-\theta_1) f_{\text{boundary}},
\end{equation*}
where $\theta_1\in  C^{\infty}([-1,1])$ and $\eta\in  C^{\infty}([0,T])$ satisfy

\begin{equation*}
\left\{\begin{array}{ll}
\theta_1(x)=1&x\in[-1,-a-\varepsilon/3],\\
\theta_1(x)=0&x\in[-a-\varepsilon/4,1],\\
0\leq\theta_1(x)\leq1&x\in[-1,1]
\end{array}\right.
~~\text{ and }~~
\left\{\begin{array}{ll}
\eta(t)=1&t\in[0,T/3],\\
\eta(t)=0&t\in[2T/3,T],\\
0\leq\eta(t)\leq1&t\in[0,T].
\end{array}\right.
\end{equation*}
For each $g\in L^2((0,T)\times\Omega)$, 
if we denote by $\varphi$ the corresponding solution to System~\eqref{eq:dual},
we have

\begin{align*}
\displaystyle\int_{(0,T)\times\Omega}f_{\text{internal}}g
&=\displaystyle\int_{(0,T)\times\Omega}
(\eta\theta_1 f_{\text{free}}+(1-\theta_1) f_{\text{boundary}})
(-\partial_t - \partial_x^2 - x^2\partial_y^2)\varphi\\
&=\displaystyle\int_{(0,T)\times\Omega}
 f_{\text{free}}(-\partial_t - \partial_x^2 - x^2\partial_y^2)(\eta\theta_1\varphi)\\
&\quad+\displaystyle\int_{(0,T)\times\Omega}f_{\text{free}}(\partial_t\eta\theta_1\varphi+\eta\partial_x^2\theta_1\varphi+2\eta\partial_x\theta_1\partial_x\varphi)\\
&\quad+\displaystyle\int_{(0,T)\times\Omega_L} f_{\text{boundary}}(-\partial_t - \partial_x^2 - x^2\partial_y^2)[(1-\theta_1)\varphi]\\
&\quad-\displaystyle\int_{(0,T)\times\Omega}f_{\text{boundary}}(\partial_x^2\theta_1\varphi+2\partial_x\theta_1\partial_x\varphi).
\end{align*}
Using Definition \ref{def:trans}, we deduce that

\begin{align*}
\displaystyle\int_{(0,T)\times\Omega}f_{\text{internal}}g
&=\langle f_0,\theta_1\varphi(0)\rangle_{V(\Omega)',V(\Omega)}
+\displaystyle\int_{(0,T)\times\Omega}f_{\text{free}}(\partial_t\eta\theta_1\varphi+\eta\partial_x^2\theta_1\varphi+2\eta\partial_x\theta_1\partial_x\varphi)\\
&\quad+\langle f_{0|\Omega_L},(1-\theta_1)\varphi(0)\rangle_{V(\Omega_L)',V(\Omega_L)}
-\displaystyle\int_{(0,T)\times\Omega}f_{\text{boundary}}(\partial_x^2\theta_1\varphi+2\partial_x\theta_1\partial_x\varphi)\\
&=\langle f_0,\varphi(0)\rangle_{L^2(\Omega)}
+\displaystyle\int_0^T\langle h,\varphi\rangle_{V(\Omega)',V(\Omega)},
\end{align*}
where 

\begin{equation*}
h\coloneqq f_{\text{free}}(\partial_t\eta\theta_1+\eta\partial_x^2\theta_1)
-2\partial_x(f_{\text{free}}\eta\partial_x\theta_1)
-f_{\text{boundary}}\partial_x^2\theta_1+2\partial_x(f_{\text{boundary}}\partial_x\theta_1)
\in L^2(0,T;V(\Omega)').
\end{equation*}
We deduce that $f_{\text{internal}}$ is the unique weak solution in 
$ C([0,T];L^2(\Omega))\cap L^2(0,T;V(\Omega))$ to

\begin{equation*}
\left\{
\begin{array}{ll}
(\partial_t - \partial_x^2 - x^2\partial_y^2)f_{\text{internal}}(t,x,y)  = h(t,x,y) &  t\in [0,T], (x,y)\in \Omega,\\
f_{\text{internal}}(t,x,y) = 0 &  t \in [0,T], (x,y)\in \partial\Omega,\\
f_{\text{internal}}(0,x,y) = f_0(x,y)& (x,y)\in \Omega
\end{array}
\right.
\end{equation*}
and satisfies $f_{\text{internal}}(T,\cdot,\cdot)=0$. The control $h$ is not regular enough --- it is in $L^2(0,T;V(\Omega)')$ instead of $L^2((0,T)\times \Omega)$ --- so we regularize it again by defining 
\[f\coloneqq\eta\theta_2 f_{\text{free}}+(1-\theta_2) f_{\text{internal}},\]
where $\theta_2\in  C^{\infty}([-1,1])$  satisfies

\begin{equation*}
\left\{\begin{array}{ll}
\theta_2(x)=1&x\in[-1,-a-\varepsilon/5],\\
\theta_2(x)=0&x\in[-a-\varepsilon/6,1],\\
0\leq\theta_2(x)\leq1&x\in[-1,1].
\end{array}\right.
\end{equation*}
Again, for each $g\in L^2((0,T)\times\Omega)$, 
if we denote by $\varphi$ the solution to System~\eqref{eq:dual}, we have

\begin{equation*}
\int_{(0,T)\times\Omega}fg
=\langle f_0,\varphi(0)\rangle_{L^2(\Omega)}
+\displaystyle\int_{(0,T)\times\Omega} u\varphi,
\end{equation*}
where 

\[u:=f_{\text{free}}(\partial_t\eta\theta_2+\eta\partial_x^2\theta_2)
-2\partial_x(f_{\text{free}}\eta\partial_x\theta_2)
-f_{\text{internal}}\partial_x^2\theta_2+2\partial_x(f_{\text{internal}}\partial_x\theta_2)
\in L^2((0,T)\times\Omega).\]
We deduce that $f$ is the unique weak solution in 
$ C([0,T];L^2(\Omega))\cap L^2(0,T;V(\Omega))$ to System~\eqref{eq:control_problem_eng}
with the control $u$ and satisfying $f(T,\cdot,\cdot)=0$. 
Moreover, we remark that $u$ is localized in $\omega$, which concludes the proof.
\end{proof}

\section{Proof of the existence of the cutoff (Lemma~\ref{th:cutoff})}\label{sec:cutoff_proof}
We consider $\tilde\gamma\in C^{0}(\set{R};\set{R}^2)$ the path $\gamma$ that we extend vertically on the top and bottom. I.e. if we note $\gamma(0) = (x_0,0)$ and $\gamma(1) = (x_1,\pi)$, we set 

\begin{equation*}
\tilde\gamma(s)= \left\{\begin{array}{ll}
\gamma(s)&\mbox{ for all }s\in[0,1],\\
(x_0,s)&\mbox{ for all }s<0,\\
(x_1,\pi+s-1)&\mbox{ for all }s>1.
\end{array}\right.
\end{equation*}
Denote by $\tilde\omega_{\text{right}}$ the connected component of $\set R^2\setminus \tilde \gamma(\set{R})$ containing $(2,0)$ (for instance), i.e. the set that is ``right of $\tilde\gamma$''. Then set $\tilde \theta := \mathbf 1_{\tilde\omega_{\text{right}}}$ the indicator function  of $\tilde \omega_{\text{right}}$. Finally, choosing a smooth mollifier $\rho_\varepsilon$ that has its support in $B(0,\varepsilon/2)$, we set $\theta: = \rho_{\varepsilon}\ast \tilde\theta$.

Since $\tilde \theta$ is locally constant outside of $\tilde \gamma(\set{R})$, $\theta$ is locally constant around each $z\in\set R^2$ such that $\distance(z,\tilde\gamma(\set{R}))>\varepsilon/2$. With the definition of $\omega_0$ (Eq.~\eqref{omega_0}), this proves that $\supp(\nabla\theta)\cap \Omega\subset \omega_0$. Now, by the definition of $a$ (Eq.~\eqref{def a}), if $z= (x,y)$ is in $\omega_{\text{right}}\setminus \omega_0$, i.e. if $a<x<1$ but $z$ is not in $\omega_0$, then the open segment $\{(x',y), x<x'<1\}$ lies outside of $\omega_0$. Thus, we can see that $z$ is in the connected component of $(2,0)$, i.e. it is in $\tilde \omega_{\text{right}}$ and since it is not in $\omega_0$, $\distance(z,\tilde\gamma(\set R))>\varepsilon$, and thus $\theta(z) = 1$. Similarly, if $z$ is in $\omega_{\text{left}}\setminus \omega_0$, then $z$ is in the connected component of $(-2,0)$, which is not in the same connected component as $(2,0)$, and thus $\theta(z) = 0$.\hfill\qed

\section{Sketch of the proof of Theorem~\ref{th:main_neg_R}}\label{sec:neg_R}
The proof of Theorem~\ref{th:main_neg_R} goes along the same lines as the proof of Theorem~\ref{th:main_negative}, but is actually simpler: the eigenfunctions are exactly the Gaussian $\eu^{-nx^2\!/2}$, with associated eigenvalue exactly $n$, and we don't need the Lemma~\ref{th:est_hol}, nor any equivalent.

Following the proof until equation~\eqref{eq:obs3.6}, with the change of variables $z_x = e^{-t+iy -x^2\!/2}$, we find that the null-controllability on $\omega$ (as defined in the statement of the Theorem) would imply that with $\mathcal D_x = \{\eu^{-T-x^2\!/2}<|z|<\eu^{-x^2\!/2}, ||\arg(z)|-y_0|<f(x)\}$,

\begin{align*}
\int_{D(0,\eu^{-T})} \Big\lvert\sum_{n>0} a_n z^{n-1}\Big\rvert^2\diff \lambda(z) &\leq 
C\int_{\set R}\int_{\mathcal D_x}\Big\lvert\sum_{n>0} a_n z^{n-1}\Big\rvert^2\diff \lambda(z)\diff x.\\
\intertext{Then, with $U = \bigcup_{x\in \set R} \mathcal D_x$,}
\int_{D(0,\eu^{-T})} \Big\lvert\sum_{n>0} a_n z^{n-1}\Big\rvert^2\diff \lambda(z) &\leq 
C\Big\lvert\sum_{n>0} a_n z^{n-1}\Big\rvert^2_{L^\infty(U)}.
\end{align*}
(This time, we only write $|p|_{L^2} \leq C|p|_{L^\infty}$.) But this does not hold. Indeed, $z_0 = \eu^{-T/2+\iu y_0}$ (for instance) is non adherent to $U$, and we can construct a counter-example to the estimate above with Runge's theorem, as in the end of the proof of theorem~\ref{th:main_neg_R}.

\bibliographystyle{plain}
\bibliography{bib.bib}
\end{appendix}
\end{document}

%% file: omega_pos_th.tex
\begin{tikzpicture}[scale=2.5]
\draw[fill=green!50] (-1.1,0) .. controls (-1.7767,0.3453) and (-1.5545,0.7659) .. (-1.3545,0.7659) .. controls (-1.1545,0.7659) and (-0.7992,0.3044) .. (0.0481,0.3126) .. controls (0.3371,0.3435) and (0.6,0.6) .. (0.8,1) .. controls (0.8,1) and (1,1) .. (1,1) .. controls (0.8,0.6) and (0.4668,0.2183) .. (0.0444,0.192) .. controls (-0.6881,0.1878) and (-0.9,0.4) .. (-1.1451,0.5424) .. controls (-1.3885,0.7097) and (-1.4759,0.6101) .. (-1.3977,0.4409) .. controls (-1.2759,0.2101) and (-1.1465,0.1925) .. (-0.9,0) .. controls (-0.9,0) and (-1.1,0) .. (-1.1,0);
\draw[dashed] (-1.488,0.55) -- (-1.488,-0) node[below]{$-a$};
  \draw (-2,0) rectangle (2,1);
 \draw[->] (-2.2,0) -- (2.2,0) node[above]{$x$};
 \draw[->] (0,-0.2) -- (0,1.2) node[right]{$y$};
\draw[->] (1.2402,0.2244)node[right]{$\omega_0$} -- (0.5407,0.4428);
\draw[blue] (-1,0) .. controls (-1.4332,0.2846) and (-1.4225,0.299) .. (-1.4743,0.4381) .. controls (-1.5265,0.6283) and (-1.4179,0.7119) .. (-1.337,0.6876) .. controls (-1.2643,0.6719) and (-1.1059,0.5927) .. (-0.817,0.4359) .. controls (-0.5284,0.2964) and (-0.4264,0.2928) .. (-0.1871,0.2626) .. controls (0.1021,0.2333) and (0.2273,0.2458) .. (0.4758,0.4386) .. controls (0.6298,0.5679) and (0.7454,0.7164) .. (0.9,1);
\draw[blue,->] (-0.5569,0.5726)node[above]{$\gamma$} -- (-0.7525,0.4054);
\end{tikzpicture}

%% file: omega_neg_th.tex
\begin{tikzpicture}[scale = 2.5]
  \draw (-2,0) rectangle (2,1);
  \draw[fill = green!50] (-0.5,0) -- (-1.8,0) .. controls 
(-2,0.2) and (-2,0.6) .. (-1.8,0.7) .. controls 
(-1.6,0.9) and (-0.6,0.9) ..  (-0.6,0.7) .. controls 
(-0.6,0.6) and (-1.4,0.5) .. (-1.4,0.2) .. controls (-1.4,0.1) .. cycle;
  \draw[fill = green!50]  (1.8,0.3) .. controls 
(2,0.3) and (2,0.4) .. (1.8,1) --  (1,1) .. controls 
(1,0.3) and (1.5,0.3) .. cycle;
  \node (a) at (-0.5,1.2) {$\omega$};
  \draw[->] (a) -- (-1.3,0.7);
  \draw[->] (a) -- (1.6,0.7); 
  \draw[thick] (-1.4,0.2) -- (1.4,0.2);
  \draw[thick, dashed] (0,0.2) -- (2,0.2) node[right]{$y_0$};
  \draw[dashed, thick] (-1.4,0) node[below left = 0 and -0.2]{$-a$} -- (-1.4,0.2);
  \draw[dashed, thick] (1.4,0) node[below right = 0.05 and -0.1]{$a$} -- 
(1.4,0.2);
 \draw[->] (-2.2,0) -- (2.2,0) node[above]{$x$};
 \draw[->] (0,-0.2) -- (0,1.2) node[right]{$y$};
\end{tikzpicture}

%% file: omega_R.tex
\begin{tikzpicture}[scale = 2.5]
\newcommand{\topboundary}{ (2,0.47) .. controls (0.1,0.5) and (0.3,0.8) .. (0,0.8) .. controls (-0.2,0.8) and (-0.1,0.5) .. (-2,0.47)}
\fill[fill=green!50] (-2,1) --  (2,1) -- \topboundary  -- cycle;
\draw \topboundary;
\newcommand{\bottomboundary}{ (2,0.43) .. controls (0.1,0.4) and (0.3,0.1) .. (0,0.1) .. controls (-0.2,0.1) and (-0.1,0.4) .. (-2,0.43)}
\fill[fill=green!50] (-2,0) --  (2,0) -- \bottomboundary -- cycle;
\draw \bottomboundary;
\draw[dashed] (-2,0.45) -- (2,0.45) node[right]{$y_0$};
  \draw (-2,0) -- (2,0) ( -2,1) -- (2,1);
 \draw[->] (-2.2,0) -- (2.2,0) node[above]{$x$};
 \draw[->] (0,-0.2) -- (0,1.2) node[right]{$y$};
\node (v1) at (0.6,1.1) {$\omega$};

\draw[->] (v1) -- (1.3,0.8);
\draw[->] (v1) -- (1.3,0.1);
\end{tikzpicture}

%% file: example1.tex
\begin{tikzpicture}[scale=2]
\draw (-1,0) -- (1,0) -- (1,1) -- (-1,1) -- (-1,0);
\draw[->] (-1.2,0) -- (1.2,0);
\draw[->] (0,-0.2) -- (0,1.2);
\draw  (-0.75,0) -- (-0.75,0.75) -- (-0.25,0.75) -- (-0.25,0) -- (-0.75,0);
\draw  (0.75,1) -- (0.75,0.25) -- (0.25,0.25) -- (0.25,1) -- (0.75,1);
\fill[opacity=0.5,color=green]  (-0.75,0) -- (-0.75,0.75) -- (-0.25,0.75) -- (-0.25,0) -- (-0.75,0);
\fill[opacity=0.5,color=green]  (0.75,1) -- (0.75,0.25) -- (0.25,0.25) -- (0.25,1) -- (0.75,1);
\path (-0.5,0.25) node {$\omega$};
\path (0.5,0.75) node {$\omega$};
\path (-0.25,-0.1) node {$-a$};
\end{tikzpicture}

%% file: example2.tex
\begin{tikzpicture}[scale=2]
\draw[->] (-1.2,0) -- (1.2,0);
\draw[->] (0,-0.2) -- (0,1.2);
\draw[dotted] (-0.5,0) -- (-0.5,0.5) ;
\draw[fill = green!50]  (-0.8,0) -- (-0.2,1) -- (-0.8,1) -- (-0.2,0) -- cycle;
\path (-0.6,0.15) node {$\omega$};
\path (-0.5,0.75) node {$\omega$};
\draw[thick, dashed] (-0.5,0.5) -- (-0.5,0) node[below] {$-a$};
\draw (-1,0) -- (1,0) -- (1,1) -- (-1,1) -- (-1,0);
\end{tikzpicture}

%% file: example3.tex
\begin{tikzpicture}[scale=2]
\draw (-1,0) -- (1,0) -- (1,1) -- (-1,1) -- (-1,0);
\draw[->] (-1.2,0) -- (1.2,0);
\draw[->] (0,-0.2) -- (0,1.2);
\draw[dashed, thick] (-0.5,0) -- (-0.5,0.375) ;
\draw[dashed, thick] (-0.75,0) -- (-0.75,0.625) ;
\draw  (-1,0) -- (-0.25,0) -- (-0.25,0.125) -- (-0.75,0.625) -- (-0.75,0.875) -- (-0.25,0.375) -- (-0.25,1) -- (-1,1) -- (-1,0);
\fill[opacity=0.5,color=gray,color=green]  (-1,0) -- (-0.25,0) -- (-0.25,0.125) -- (-0.75,0.625) -- (-0.75,0.875) -- (-0.25,0.375) -- (-0.25,1) -- (-1,1) -- (-1,0);
\path (-0.45,0.8) node {$\omega$};
\path (-0.5,-0.1) node {$-a$};
\path (-0.75,-0.1) node {$-b$};
\end{tikzpicture}

%% file: omega0_left_right.tex
\begin{tikzpicture}[scale=2.5]
\draw[fill=green!50] (-1.1,0) .. controls (-1.7767,0.3453) and (-1.5545,0.7659) .. (-1.3545,0.7659) .. controls (-1.1545,0.7659) and (-0.7992,0.3044) .. (0.0481,0.3126) .. controls (0.3371,0.3435) and (0.6,0.6) .. (0.8,1) .. controls (0.8,1) and (1,1) .. (1,1) .. controls (0.8,0.6) and (0.4668,0.2183) .. (0.0444,0.192) .. controls (-0.6881,0.1878) and (-0.9,0.4) .. (-1.1451,0.5424) .. controls (-1.3885,0.7097) and (-1.4759,0.6101) .. (-1.3977,0.4409) .. controls (-1.2759,0.2101) and (-1.1465,0.1925) .. (-0.9,0) .. controls (-0.9,0) and (-1.1,0) .. (-1.1,0);
\node[below] at (-1.488,-0) {$-a$};
\node[below] at (1.488,-0) {$a$};

  \draw (-2,0) rectangle (2,1);
 \draw[->] (-2.2,0) -- (2.2,0) node[above]{$x$};
 \draw[->] (0,-0.2) -- (0,1.2) node[right]{$y$};
\draw[->] (0.5,0.15)node[right]{$\omega_0$} -- (0.5407,0.4428);
\draw[blue] (-1,0) .. controls (-1.4332,0.2846) and (-1.4225,0.299) .. (-1.4743,0.4381) .. controls (-1.5265,0.6283) and (-1.4179,0.7119) .. (-1.337,0.6876) .. controls (-1.2643,0.6719) and (-1.1059,0.5927) .. (-0.817,0.4359) .. controls (-0.5284,0.2964) and (-0.4264,0.2928) .. (-0.1871,0.2626) .. controls (0.0494,0.2351) and (0.2062,0.2708) .. (0.4758,0.4386) .. controls (0.6298,0.5679) and (0.7454,0.7164) .. (0.9,1);
\draw[blue,->] (-0.5569,0.5726)node[above]{$\gamma$} -- (-0.7525,0.4054);
\filldraw[opacity=0.75,pattern=dots,pattern color=red]  (-2,0) -- (-1.488,0) -- (-1.488,1) -- (-2,1) -- (-2,0);
\filldraw[opacity=0.75,pattern=dots,pattern color=blue]  (1.488,0) -- (2,0) -- (2,1) -- (1.488,1) -- (1.488,0);

\node[below, fill=white, fill opacity = 0.6, text opacity = 1] at (-1.7,0.4) {$\omega_{\text{left}}$};
\node[below, fill=white, fill opacity = 0.6, text opacity = 1] at (1.7,0.5) {$\omega_{\text{right}}$};
\end{tikzpicture}

%% file: omega_neg.tex
\begin{tikzpicture}[scale = 2.5]
  \draw (-2,0) rectangle (2,1);
  \draw[blue, fill=blue!50, thick] (-1.2,0.1) rectangle (1.2,0.3);
  \draw[fill = green!50] (-0.5,0) -- (-1.8,0) .. controls 
(-2,0.2) and (-2,0.6) .. (-1.8,0.7) .. controls 
(-1.6,0.9) and (-0.6,0.9) ..  (-0.6,0.7) .. controls 
(-0.6,0.6) and (-1.4,0.5) .. (-1.4,0.2) .. controls (-1.4,0.1) .. cycle;
  \draw[fill = green!50]  (1.8,0.3) .. controls 
(2,0.3) and (2,0.4) .. (1.8,1) --  (1,1) .. controls 
(1,0.3) and (1.5,0.3) .. cycle;
  \node (a) at (-0.5,1.2) {$\omega$};
  \draw[->] (a) -- (-1.3,0.7);
  \draw[->] (a) -- (1.6,0.7); 
  \draw[thick] (-1.4,0.2) -- (1.4,0.2);
  \draw[thick, dashed] (0,0.2) -- (2,0.2) node[right]{$y_0$};
  \draw[dashed, thick] (-1.4,0) node[below left = 0 and -0.2]{$-a$} -- (-1.4,0.2);
  \draw[dashed, thick] (1.4,0) node[below right = 0.05 and -0.1]{$a$} -- 
(1.4,0.2);
  \draw[dashed, thick, blue] (-1.2,0) node[below = -0.05]{$-a'$} -- 
(-1.2,0.2);
  \draw[dashed, thick, blue] (1.2,0) node[below = -0.05]{$a'$} -- 
(1.2,0.2);
 \draw[->] (-2.2,0) -- (2.2,0) node[above]{$x$};
 \draw[->] (0,-0.2) -- (0,1.2) node[right]{$y$};
\end{tikzpicture}

%% file: U.tex
\begin{tikzpicture}[scale=2.1]
\draw[fill=yellow!50]  (27:0.4) arc(27:45:0.4) -- (45:1.1) arc(45:315:1.1) -- 
(315:0.4) arc(315:333:0.4) -- (333:1.1) arc(333:387:1.1) -- cycle;
\node at(-0.4,0.4){$U$};
\draw[thick,->](-1.2,0) -- (1.2,0);
\draw[thick,->] (0,-1.2) -- (0,1.2);
\draw[thick] (0,0) -- (36:1.3);
\draw[thick, ->] (0:0.6) arc(0:36:0.6) node[pos=0.5, right]{$y_0$};
\draw[thin, gray] (0,0) circle[radius = 0.4];
\draw[<->] (0,0) -- (-110:0.4) node[pos = 0.5, left, 
fill=yellow!50, text=black, fill opacity = 0.5, inner sep 
= 0pt, text opacity = 1]{$e^{-(1-2\epsilon)a'^2\!/2}$};
\end{tikzpicture}

%% file: omega_symetrique.tex
\begin{tikzpicture}[scale = 1.5]
  \draw[fill = green!40] (-2,-1) rectangle (2,1);
  \draw[blue, thick,fill=white] (-1.2,0.1) rectangle (1.2,0.3);
  \draw[blue, thick,fill=white] (-1.2,-0.1) rectangle (1.2,-0.3);
 \node[fill=white, fill opacity = 0.7, text opacity = 1, inner sep = 3pt] 
at(1.4,0.6) {$\tilde \omega$};
 \draw[->] (-2.2,0) -- (2.2,0) node[above]{$x$};
 \draw[->] (0,-1.2) -- (0,1.2) node[left]{$y$};
 
 \draw[green!60!black, very thick] (-1.5,1) -- (-1.5,-1);
 \draw[very thick, black!80]
  (-1.5,1.1) node[above]{$|x|>a'$} -- (-1.5,1) 
  (-1.5,-1) -- (-1.5,-1.1) node(complet){};
 \draw[green!60!black, very thick] (0.7,1) -- (0.7,0.3) 
  (0.7,0.1) --(0.7,-0.1)
  (0.7,-0.3) -- (0.7,-1);
 \draw[very thick, black!80] 
  (0.7,1.1) node[above]{$|x|<a'$} -- (0.7,1)
  (0.7,0.3) -- (0.7,0.1)
  (0.7,-0.1) -- (0.7,-0.3)
  (0.7,-1) -- (0.7,-1.1) node(partiel){};
  \draw[dashed, thick, blue] (-1.2,-1) node[below = -0.05]{$-a'$} -- 
(-1.2,-0.2);
  \draw[dashed, thick, blue] (1.2,-1) node[below = -0.05]{$a'$} -- 
(1.2,-0.2);
 
 \begin{scope}[shift = {(-1.2,-2.6)}]
  \draw[very thick, green!70!black, fill = yellow!70!green] (0,0) circle[radius = 0.47];
  \draw[fill = white] (0,0) circle[radius = 0.14];
  \draw (complet.center) -- (100:0.47);
  \node at (-0.2,0.2) [fill = white, fill opacity = 0.5, inner sep = 0.1pt, text opacity = 1] {$\mathcal D_x$};
 \draw[->] (-1.1,0) -- (1.1,0);
 \draw[->] (0,-1.1) -- (0,1.1);
 \end{scope}
 
 \begin{scope}[shift={(1.2,-2.6)}]
  \draw[fill=yellow!70!green]  (54:0.3) -- (54:1) arc(54:306:1) -- 
(306:0.3) arc(306:54:0.3) -- cycle;
  \draw[fill=yellow!70!green]  (-18:0.3) -- (-18:1) arc(-18:18:1) -- 
(18:0.3) arc(18:-18:0.3) -- cycle;
  \draw[very thick, green!70!black]  (54:1) arc(54:306:1)  (342:1)arc(342:378:1);
  \draw (partiel.center) -- (100:1);
  \node at (-0.45,0.45) [fill = white, fill opacity = 0.5, inner sep = 0.1pt, text 
opacity = 1] {$\mathcal D_x$};
 \draw[->] (-1.1,0) -- (1.1,0);
 \draw[->] (0,-1.1) -- (0,1.1);
 \end{scope}

\end{tikzpicture}

%% file: D.tex
\begin{tikzpicture}[scale=2.5]
\draw[fill=yellow!50]  (27:0.4) arc(27:45:0.4) -- (45:1.1) arc(45:315:1.1) -- 
(315:0.4) arc(315:333:0.4) -- (333:1.1) arc(333:387:1.1) -- cycle;
\draw[fill=yellow!70!green]  (18:0.3) arc(18:54:0.3) -- (54:1) arc(54:306:1) -- 
(306:0.3) arc(306:342:0.3) -- (342:1) arc(342:378:1) -- cycle;
\draw[fill=yellow!50] (0,0) circle[radius = 0.1];
\draw[->] (-0.9,0.9) node[above]{$U$} -- (145:1.05);
\draw[gray, thin] (0,0) circle[radius = 0.3];
\draw[<->] (0,0) -- (-110:0.3) node[pos = 0.5, left, 
fill=yellow!75!green, text=black, fill opacity = 0.5, inner sep 
= 0pt, text opacity = 1]{$e^{-(1-\epsilon)a'^2\!/2}$};
\draw[thick,->](-1.2,0) -- (1.2,0);
\draw[thick,->] (0,-1.2) -- (0,1.2);
\node at(-0.3,0.45)[fill=white, fill opacity = 0.5, text opacity = 1, inner sep 
= 0.5pt]{$K = \overline{\bigcup_{x=-1}^1 \mathcal D_x}$};
\end{tikzpicture}

%% file: runge.tex
\begin{tikzpicture}[scale=2]
\draw[fill=yellow!50]  (27:0.4) arc(27:45:0.4) -- (45:1.1) arc(45:315:1.1) -- 
(315:0.4) arc(315:333:0.4) -- (333:1.1) arc(333:387:1.1) -- cycle;
\node at(-0.6,0.6)[fill=white, fill opacity = 0.5, text opacity = 1, inner sep 
= 2pt]{$U$};
\draw[fill=red, fill opacity = 0.6] (0,0) circle[radius=0.6];
\draw[red!80!black, thick, ->] (-1,-0.9) node[below]{$D(0,e^{-T})$} -- (-0.2,-0.2);
\draw[blue, thick, arrows={Rays[n=4]-}] (36:0.5) node[below,xshift = 3,fill=white, 
fill opacity = 0.4, text opacity = 1, inner sep = 1pt]{$z_0$} -- 
(36:1.4);
\draw[thick,->](-1.2,0) -- (1.2,0);
\draw[thick,->] (0,-1.2) -- (0,1.2);
\end{tikzpicture}

%% file: omega_corollary.tex
\begin{tikzpicture}[scale=2.5]
\draw[fill=green!50] (-0.4,0) .. controls (-0.6,0.1) and (-1,0.1) .. (-1,0.2)  .. controls (-1,0.6) and (0.5,0.8) .. (0.5,1)node[pos = 0.5](v3){} -- (-1,1) .. controls (-0.7,0.9) and (-0.7,0.9) .. (-0.7,0.8) .. controls (-0.7,0.7) and (-1.2,0.5) .. (-1.8,0)node[pos = 0.5](v2){} -- cycle;
\draw[dashed] (-1,0.2) node (v1) {} |- (0,0) node[pos=0.5,below]{$-a$};
  \draw (-2,0) rectangle (2,1);
 \draw[->] (-2.2,0) -- (2.2,0) node[above]{$x$};
 \draw[->] (0,-0.2) -- (0,1.2) node[right]{$y$};
\draw (v1.center) -| (0,0) node[pos = 0.5,right]{$y_0$};
\node[fill = white, fill opacity = 0.4, text opacity = 1] at (-0.5,0.8) {$\omega$};
\draw[->] (-1.5,0.7)node[left]{$\gamma_1$} -- (v2.center);
\draw[->] (0.5,0.5) node[right]{$\gamma_2$} -- (v3.center);
\draw[blue, thick] (-0.8,0) .. controls (-1,0.1) and (-1.05,0.1) .. (-1.05,0.2)  .. controls (-1,0.6) and (0,0.8) .. (0,1) node[pos = 0](v4) {};
\draw[->, blue] (-1.6,0.3)node[left]{$\gamma$} -- (v4.center);
\end{tikzpicture}